\documentclass[reqno,11pt]{amsart}
\usepackage{graphicx, amsmath, amssymb, amscd, mathtools, hyperref, amsthm, euscript, 
amsfonts,bm,color,bbm}  

\usepackage{mathtools}
\usepackage{amsmath}
\usepackage{amssymb}
\usepackage{yhmath}
\usepackage{graphicx}
\usepackage{ mathrsfs }
\usepackage{bbm}
\usepackage{xcolor}
\usepackage{tikz-cd}
\DeclareMathAlphabet{\mathpzc}{OT1}{pzc}{m}{it}

\newtheorem{theorem}{Theorem}[section]

\newtheorem{thm}[equation]{Theorem}

\newtheorem{Q}[equation]{Question}

\newtheorem{prop}[equation]{Proposition}

\newtheorem{cor}[equation]{Corollary}
\newtheorem{lem}[equation]{Lemma}

\newtheorem{dfn}[equation]{Definition}
\theoremstyle{remark}
\newtheorem{rmk}[theorem]{Remark}
\newtheorem{Rmk}[theorem]{Remark}
\numberwithin{equation}{section}

\numberwithin{equation}{section}

\newcommand{\be}{begin{equation}}
\newcommand{\dw}{d_{\mathsf w}}

\newcommand{\e}{\epsilon}

\renewcommand{\c}{\mathbb{C}}
\newcommand{\br}{\mathbb{R}}

\newcommand{{\grinv}}{{\Cal G}^{-r}}

\newcommand{\ba}{\backslash}

\newcommand{\G}{\Gamma}

\newcommand{\Cal}{\mathcal}

\newcommand{\la}{\langle}

\newcommand{\bp}{\begin{pmatrix}}
\newcommand{\ep}{\end{pmatrix}}
\renewcommand{\be}{\begin{equation}}
\newcommand{\ee}{\end{equation}}

\renewcommand{\bp}{{\rm bp}}

\newcommand{\SO}{\operatorname{SO}}

\renewcommand{\L}{\Cal L}

\newcommand{\PSL}{\op{PSL}}

\newcommand{\norm}[1]{\lVert #1 \rVert}

\newcommand{\op}{\operatorname}

\newcommand{\BMS}{\operatorname{BMS}}

\newcommand{\rank}{\operatorname{rank}}

\newcommand{\Om}{\Omega}

\def\om{\omega}

\newcommand{\ga}{\gamma}
\newcommand{\F}{\mathcal F}
\newcommand{\La}{\Lambda}
\renewcommand{\i}{\op{i}}

\def\e{\mathrm{e}}
\def\i{\mathrm{i}}

\def\dim{\operatorname{dim}}

\def\SO{\operatorname{SO}}

\def\PSL{\operatorname{PSL}}

\newcommand{\pc}{P^{\circ}}

\newcommand{\fg}{\mathfrak{g}}

\newcommand{\fa}{\mathfrak{a}}

\newcommand{\Ga}{\Gamma}
\newcommand{\cal}{\mathcal}
\renewcommand{\e}{\varepsilon}
\renewcommand{\epsilon}{\e}

\newcommand{\inte}{\op{int}}

\renewcommand{\la}{\lambda}

\newcommand{\m}{\mathsf m}
\newcommand{\E}{\mathcal E}
\begin{document}

\title[Dichotomy and measures]{Dichotomy and  measures
on limit sets of Anosov groups.}

\author{Minju Lee and Hee Oh}
\address{School of mathematics, Institute for Advanced Study, Princeton, NJ 08540, current address: Mathematics department, University of Chicago, Chicago, IL, USA}
\address{Mathematics department, Yale university, New Haven, CT 06520 and Korea Institute for Advanced Study, Seoul, Korea}
\address{}

\email{minju1@uchicago.edu}
\email{hee.oh@yale.edu}
\thanks{Lee and Oh are respectively partially supported by 
 the NSF grant No. DMS-1926686 (via the
Institute for Advanced Study), and the NSF grant No. DMS-1900101.}

\begin{abstract}  Let $G$ be a connected semisimple real algebraic group. For a Zariski dense Anosov subgroup $\Ga<G$,
we show that a $\Ga$-conformal measure is supported on the limit set of $\Ga$ if and only if its {\it dimension} is $\Ga$-critical. This implies the uniqueness of a $\G$-conformal measure for each critical dimension, answering the question posed in our earlier paper with Edwards \cite{ELO2}.
We obtain this by proving a higher rank analogue of the Hopf-Tsuji-Sullivan dichotomy for the {\it maximal diagonal} action. Other applications include an analogue of the Ahlfors measure conjecture for Anosov subgroups. \end{abstract}

\maketitle

\section{Introduction}
Let $G$ be a connected semisimple real algebraic group. In this paper, we investigate properties of
$\Gamma$-conformal measures on the Furstenberg boundary of $G$ for a certain class of 
discrete subgroups $\Gamma$ of $G$, called {\it Anosov} subgroups.
Associated to each conformal measure 
does there exist a linear form
on the Cartan subspace of the Lie algebra of $G$, which may be regarded as the {\it dimension} of the measure. We show that a $\Ga$-conformal measure is supported on the limit set
of $\G$ if and only if this dimension is $\Ga$-critical. We deduce  this result from a higher rank analogue of the Hopf-Tsuji-Sullivan dichotomy for the maximal diagonal action, which relates the supports of conformal measures, critical exponents of Poincare series, and
the dynamical properties of the action of a maximal diagonal subgroup
on $\Ga\ba G$ relative to higher rank generalizations of Bowen-Margulis-Sullivan measures. Applications include an analogue of the Ahlfors measure conjecture for Anosov subgroups of $G$.

\medskip 

To state our main results precisely, we let  $P=MAN$ be a minimal parabolic subgroup of $G$ with a fixed Langlands decomposition, where  $A$ is a maximal real split torus of $G$, $M$
is the maximal compact subgroup centralizing
$A$ and $N$ is the unipotent radical of $P$. Let $\fg=\op{Lie} G$, $\fa=\op{Lie} A$ and $\fa^+$ denote the positive Weyl chamber so that $\log N$ consists of positive root subspaces. Let $K$ be a maximal compact subgroup so that the Cartan decomposition $G=K(\exp \fa^+) K $ holds.
Let $\mu:G\to \fa^+$ denote the Cartan projection map defined by the condition
$\exp \mu(g)\in KgK$ for all $g\in G$.

A finitely generated discrete subgroup $\Ga<G$ is called an \textit{Anosov subgroup} (with respect to $P$) if  there exist constants $C,C'>0$ such that for all $\ga\in\Ga$ and all simple root $\alpha$ of $(\fg, \fa)$,
$$
\alpha(\mu(\ga))\ge C|\ga|-C'
$$
 where $|\ga|$ denotes the word length of $\ga$ with respect to a fixed finite symmetric set of generators of $\Ga$.
 The notion of Anosov subgroups was first introduced by Labourie for surface groups \cite{La}, and was extended to general word hyperbolic groups by Guichard-Wienhard \cite{GW}.
Several equivalent characterizations have been  established, one of which is the above definition (see \cite{GGKW}, \cite{KL}, \cite{KLP1}, \cite{KLP2}). Anosov subgroups are regarded as natural generalizations of convex cocompact subgroups of rank one groups.


\subsection*{Uniqueness of conformal measures} We set $\F:=G/P$ which is the Furstenberg boundary of $G$. Let $\G<G$ be a Zariski dense discrete subgroup.
A Borel probability measure $\nu$ on $\F$ is called a $\Gamma$-conformal measure if, 
there exists a linear form $\psi\in \fa^*$ such that
for any $\ga\in \G$ and $\xi\in \F$,
 \be\label{gc0} \frac{d\ga_* \nu}{d\nu}(\xi) =e^{\psi (\beta_\xi (e, \gamma))}\ee
where $\beta$ denotes the $\frak a$-valued Busemann function defined in Def. \eqref{Bu}. 
We call $\nu$ a $(\Ga, \psi)$-conformal measure and $\psi$ the  dimension of $\nu$. Although $\psi$ is a linear form instead of a number, we find it convenient to treat it as a sort of dimension of the measure $\nu$ and hence the name.

If $\rho$ denotes the half sum of all positive roots of $(\fg, \fa)$,
 the $K$-invariant probability measure on $\F$ (the Lebesgue measure) is  the unique $G$-conformal measure of dimension $2\rho$ \cite{Q7}.

We let $\psi_\Ga:\fa\to \br\cup \{-\infty\}$ denote the growth indicator function of $\Ga$ (see Def. \eqref{def.GI}). Let $\L\subset \fa^+$ denote the limit cone of $\Ga$, which is the asymptotic cone of the Cartan projection of $\Ga$.

We mention that the dimension of a $\G$-conformal measure is always bounded below by $\psi_\Ga$ \cite{Q}.
We call a linear form $\psi\in \fa^*$ {\it $\G$-critical}, or simply, critical,
if it is tangent to $\psi_\Ga$, i.e., 
$$\psi\ge \psi_\Ga\quad\text{ and }\quad  \psi(u)=\psi_\Ga(u)\;\; \text{ for some  $u\in \mathcal L \cap  \op{int} \fa^+$} .$$
When $G$ has rank one, $\psi_\Ga$ is simply the critical exponent $\delta$ of $\Ga$ and hence a critical linear form is just given by $\delta$.
Note that the dimension $\psi$ of a $\G$-conformal measure is either critical
or $\psi>\psi_\Ga$.

We denote 
by $\La$ the limit set of $\Ga$, which is the unique $\Ga$-minimal subset of $\F$. For each $\G$-critical 
dimension
$\psi\in \fa^*$, Quint  constructed a $(\Ga, \psi)$-conformal measure supported on the limit set $\La$, following the approach of Patterson and Sullivan (\cite{Pa}, \cite{Su}, \cite{Q}). 
Moreover, for any Anosov subgroup of {\it the second kind} (see \cite[Def. 5.1]{EO}), a $(\Ga, \psi)$-conformal measure exists for any dimension $\psi \ge \max (\psi_\Ga, \rho)$ by \cite[Cor. 5.3]{EO}.

Our first theorem gives a criterion on the support of a conformal measure in terms of its dimension. This generalizes Sullivan's theorem \cite{Su} that for $\Ga<\SO(n,1)$ convex cocompact,
any $\Ga$-conformal measure
of dimension equal to the critical exponent is necessarily supported on the limit set.
\begin{thm} \label{main}  
Let $\G<G$ be a Zariski dense Anosov subgroup. 
For any $\Ga$-conformal measure $\la$ on $\F$,
we have 
\begin{equation*} \la(\La)= \begin{cases} 1 & \text{if its dimension is $\Ga$-critical }\\
0 &\text{ otherwise}.
\end{cases}
\end{equation*}
In particular, for each $\G$-critical linear form $\psi\in \fa^*$, there exists a unique $\Ga$-conformal measure on $\F$ with dimension $\psi$.
\end{thm}

The second part
follows from the first part together with the result in \cite{LO}, 
which showed that there exists a unique $\Ga$ supported measure {\it supported  on $\La$} for each critical dimension. These results together also imply that
the space of all $\G$-conformal measures on $\F$ is homeomorphic to the space of directions in the interior of the limit cone of $\G$. It also follows from \cite[Thm. 10.20]{LO} that conformal measures of distinct critical dimensions are mutually singular to each other. 
The study of $\Ga$-conformal measures is directly related to the study of
positive joint eigenfunctions on the associated locally symmetric manifold $\Ga\ba G/K$
 for the ring of $G$-invariant differential operators (\cite{Su2}, \cite{EO}).

\begin{Rmk} When the rank of $G$ is at most $3$, it was proved in \cite{ELO2}
that any conformal measure of critical dimension is supported on $\La$, and the general case was posed as an open problem there (see Remark \ref{rank3}).
\end{Rmk}

\subsection*{Analogue of the Ahlfors measure conjecture} The Ahlfors measure conjecture \cite{Ah} says that the limit set of
a finitely generated discrete subgroup of $\PSL_2(\c)$ is either
$\mathbb S^2$ or has Lebesgue measure zero; this is now a theorem following from the works of Agol \cite{Ag}, Calegari-Gabai \cite{CG} and Canary \cite{Ca}.
The following theorem is analogous to the case of
Ahlfors' conjecture proved by Ahlfors himself
for convex cocompact subgroups \cite{Ah}. We denote by $\op{Leb}$ the Lebesgue measure on $\F$.
\begin{thm}\label{null} For any Zariski dense Anosov subgroup $\Ga< G$, 
we have either
$$\La=\F\quad\text{ or }\quad \op{Leb}(\La)=0.$$

In the former case, $\rank (G)=1$ and $\Ga$ is cocompact in $G$. 
\end{thm}

\medskip

\subsection*{Higher rank analogue of the Hopf-Tsuji-Sullivan dichotomy} Both theorems are deduced from a higher rank analogue of the Hopf-Tsuji-Sullivan dichotomy for the action of the maximal diagonal subgroup $A$. To state this dichotomy, we need to introduce some notations first.
Letting $\F^{(2)}$ denote the unique open diagonal $G$-orbit in $\cal F\times \cal F$, the quotient space $G/M$ is homeomorphic to
$\F^{(2)}\times \fa$ via the Hopf parameterization.  The notation $\i$ denotes the opposition involution of $\fa$, and let $db$ denote the Lebesgue measure on $\fa$.
For a given pair of $\G$-conformal measures $\la_{\psi}$ and $\la_{\psi\circ \i}$ on $\F$ with respect to $\psi$ and
$\psi\circ \i$ respectively,  one can use the Hopf parameterization
to define a non-zero $A$-invariant Borel measure $\m_{\la_{\psi}, \la_{\psi\circ {\i}}}$ on the quotient space $\G\ba G/M$, which is locally equivalent to $d\la_{\psi}\otimes d\la_{\psi\circ {\i}}\otimes db$ in the Hopf coordinates.
We will call it the Bowen-Margulis-Sullivan measure (or simply 
$\BMS$ 
measure) associated to the pair $(\la_\psi, \la_{\psi\circ \i})$. 
Each BMS measure $\m_{\la_{\psi}, \la_{\psi\circ {\i}}}$ on $\Ga\ba G/M$ can be considered as an $AM$-invariant measure on $\Ga\ba G$, for which we will use the same notation. For example, for $\psi=2\rho=\psi\circ \i$, the corresponding measure $\m_{\la_{2\rho}, \la_{2\rho}}$ is a $G$-invariant  measure on $\Ga\ba G$.

The conical limit set of $\Ga$ is defined as
$$\La_c=\{gP\in  \cal F: \text{$gA^+$ accumulates on $\Ga\ba G$}\},$$
in other words, $\La_c=\{gP\in  \cal F: \limsup \Gamma g A^+\ne \emptyset\}$,\footnote{For a sequence $S_n$ of subsets of a topological space $X$, $\limsup S_n$ is defined as the set
of all possible limits $s=\lim_{i\to \infty} s_{n_i}$ in $X$ where $s_{n_i}\in S_{n_i}$ for some infinite sequence $n_i$.
}
where $A^+=\exp \fa^+$. 
For Anosov subgroups, we have 
$$\La=\La_c,$$ as proved in \cite{KLP1} using the Morse property.

For $\psi\in \fa^*$, let $\cal M_\psi$ denote the collection of all $(\Ga, \psi)$-conformal measures. 

\begin{thm}[Dichotomy for the maximal diagonal action] \label{mp} Let $\Ga$ be a Zariski dense Anosov subgroup of $G$.
Let $\psi\in \fa^*$ be such that $\cal M_\psi\ne \emptyset$.  
Then the following are all equivalent to each other:
\begin{enumerate}
    \item $\sum_{\ga\in \Ga}e^{-\psi(\mu(\ga))} =\infty$ $($resp. $\sum_{\ga\in \Ga}e^{-\psi(\mu(\ga))} <\infty)$;
    \item $\psi$ is $\G$-critical (resp. $\psi>\psi_\Ga$);
    \item for any $\la_\psi\in \cal M_\psi$, $\lambda_\psi(\La_c)>0$  $($resp. $\lambda_\psi(\La_c)=0)$;
    \item for any $\la_\psi\in \cal M_\psi$, $\lambda_\psi(\La_c)=1$ $($resp. $\lambda_\psi(\La_c)=0)$;
    \item for any $(\la_\psi, \la_{\psi\circ \i})\in \cal M_\psi\times  \cal M_{\psi\circ\i}$,  the diagonal $\Ga$-action on $(\F^{(2)},  \lambda_\psi\otimes \lambda_{\psi\circ \i} |_{\F^{(2)}})$ is ergodic and completely conservative $($resp. non-ergodic and completely dissipative$)$; 
    \item for any $(\la_\psi, \la_{\psi\circ \i})\in \cal M_\psi\times  \cal M_{\psi\circ\i}$, the $A$-action on $(\Ga\ba G/M,  \m_{\la_\psi, \la_{\psi \circ \i}})$ is ergodic and completely conservative $($resp. non-ergodic and completely dissipative$)$;
    \item for any $(\la_\psi, \la_{\psi\circ \i})\in \cal M_\psi\times  \cal M_{\psi\circ\i}$ and any $\pc$-minimal subset $\E_0$ of $\Ga\ba G$, the $A$-action on $(\cal E_0, \m_{\la_\psi, \la_{\psi \circ \i}}|_{\cal E_0})$ is ergodic and completely conservative (resp. either $\m_{\la_\psi, \la_{\psi \circ \i}}(\cal E_0)=0$, or non-ergodic and completely dissipative$)$.
\end{enumerate}    
  \end{thm}

In the rank one case, the $A$-action on $\Ga\ba G/M$ corresponds to the geodesic flow
on the unit tangent bundle of the locally symmetric manifold $\Ga\ba G/K$. Therefore this theorem generalizes the Hopf-Tsuji-Sullivan dichotomy for the geodesic flow in the rank one case
(\cite{Ts}, \cite{Su}, \cite{Su1}, \cite{Hopf}, \cite{AS}, \cite{CI}, \cite{Ni});
we refer to Roblin's article \cite{Rob} for the most comprehensive exposition.

Theorem \ref{null} is deduced from Theorem \ref{mp} and Theorem \ref{inf} proved by Quint \cite{Q3}, using the matrix coefficient bounds for higher rank simple algebraic groups in \cite{Oh}. This in turn implies that, unless $\Ga\ba G$ is compact, $2\rho$ is not $\Ga$-critical  and hence the Haar measure on $\Ga\ba G$ is non-ergodic for the $AM$-action.

Since there exists a $\Ga$-conformal measure supported on $\La$
for each critical dimension, Theorem \ref{main} immediately follows from Theorem \ref{mp}
together with the uniqueness of $\Ga$-conformal measures supported on $\La$ \cite[Thm. 7.9]{ELO}.

\begin{rmk}\label{rank3} \rm When $\op{rank}G$ is at most $ 3$, 
it was shown in \cite{ELO2} that any $(\Ga,\psi)$-conformal measure is supported on the $u$-directional conical limit set $\La_u$ where $u$ is the unique unit vector $\psi(u)=\psi_\Ga(u)$; this implies Theorem \ref{main}. The proof of this result was based on  the Hopf-Tsuji-Sullivan Dichotomy for {\it one dimensional} diagonal flows $\{\exp (tu):t\in \br\}$ as established in \cite{BLLO}.  When the rank of $G$ exceeds $3$,
{\it directional} conical limit sets have negligible conformal measures, and hence this result of \cite{ELO2} did not prove Theorem \ref{main}. We note that while the dichotomy for one dimensional diagonal flows was obtained for any Zariski dense discrete subgroup,
our proof of Theorem \ref{mp} is heavily based on the hypothesis that $\Ga$ is Anosov.
\end{rmk}

While some of the implications of Theorem \ref{mp} were previously obtained in (\cite{LO}, \cite{LO2}),
the implication $(1)\Rightarrow (3)$ is the main new result of this paper,
which is needed for the application to Theorem \ref{main}.
Fixing  a $(\Ga, \psi)$-conformal measure
$\la_\psi$ for a critical $\psi\in \fa^*$, we consider the generalized Bowen-Margulis-Sullivan measure $\mathsf m=m_{\la_\psi,\la_{\psi\circ \i}}^{\BMS}$ on $\Ga\ba G$
for some conformal measure $\la_{\psi \circ \i}$ of dimension $\psi \circ \i$
 (see \eqref{eq.BMS0} for the definition).  
We use a variant of the Borel-Cantelli lemma for the $A^+$ action (Lemma \ref{lem.AS}) by relating the correlations functions of $\m$
with the Poincare series $\sum_{\ga\in\Ga,
    \norm{\mu(\ga)}\leq T    }e^{-\psi(\mu(\ga))}$.
This requires a control on the multiplicity of certain shadows (Lemma \ref{sm}),
the proof of which uses the following property of Anosov subgroups that for any $x\in \Ga\ba G$, accumulations of an orbit $xA$ in $\Ga\ba G$ can occur only via sequences in
$A^+\cup w_0A^+w_0^{-1}$ where $w_0$ is the longest Weyl element. In other words,
for any other Weyl element $w\ne e, w_0$, the subset $xwA^+w^{-1}$ is a proper embedding of $wA^+w^{-1}$, as was first observed in \cite[Lem. 8.13]{LO}. See
Lemmas \ref{proper} and \ref{lem.RS}. This phenomenon makes this higher rank situation a bit more like a rank one situation where the one dimensional 
subgroup $A$ is simply the union $A^+\cup w_0A^+w_0^{-1}$. Based on this and other properties of Anosov subgroups, we are able to extend the rank one argument in \cite{Rob} to this higher rank Anosov setting.

In a higher rank simple algebraic group, the conical limit set has Lebesgue measure zero for a discrete subgroup of infinite co-volume (see Proposition \ref{ds}).
We end the introduction by the following question:
 \begin{Q}Let $G$ be a connected simple real algebraic group with rank at least $2$ and $\Gamma <G$ be a Zariski dense
 discrete subgroup. Is the following true?:$$\La=\F\quad \text{ if and only if  }\quad\text{$\Gamma$ is a a lattice in $G$}.$$ \end{Q}

We remark that $\La=\F$ is equivalent to the minimality of the $P$-action on $\Ga\ba G$, which means that every $P$-orbit is dense in $\Ga\ba G$. Hence a weaker (still unknown) question than the above is
whether $\Ga$ is necessarily a lattice if the $NM$-action is minimal on $\Ga\ba G$, or equivalently if $\F$ is equal to the set of horospherical limit points  of $\G$, in the sense of \cite{LaO}.
In view of a theorem of Fraczyk and Gelander \cite{FG}, one can also ask whether the infinite injectivity radius of $\Ga\ba G$ implies that $\La$ cannot be all of $\F$ in higher rank setting.

\subsection*{Organization}
In section 2, basic definitions and properties of Anosov subgroups will be recalled.
In section 3, we prove a uniform bound on the multiplicity of
certain shadows, which is a main technical ingredient.
In section 4, we show that if $\sum_{\ga\in \Ga}e^{-\psi(\mu(\ga))}=\infty$,
then for a large compact subset $Q\subset \Ga \ba G$,
the events  $P_a=Q\cap Q a^{-1}$, $a\in A^+$ do not have a strong correlation with respect to the BMS measures of the form
$m_{\la_\psi,\la_{\psi\circ \i}}^{\BMS}$;
this will be used as a main input for the Borel-Cantelli lemma in section 5 to show that any $(\Ga,\psi)$-conformal measure is necessarily supported on the conical limit set $\La_c$.
In section 6, we establish all the equivalences of Theorem \ref{mp}. In section 7, we prove Theorem \ref{null}.

\subsection*{Acknowledgements} We would like to thank Peter Sarnak  and David Fisher for useful comments.

\section{Preliminaries}\label{s1}
Let $G$ be a connected semsimple real algebraic group. We let $P=MAN$, $\fg,\fa,\fa^+$, etc, be as defined in the introduction. 
We fix a maximal compact subgroup $K<G$ so that the Cartan decomposition $G=K(\exp \fa^+) K$ holds. Denote by $\mu:G\to \fa^+$ the Cartan projection, i.e., for $g\in G$, its Cartan projection $\mu(g)\in \fa^+$ is the unique element such that 
\be\label{Car} g\in K\exp \mu(g)K.\ee 
  We fix a norm $\|\cdot\|$ on $\fa$
which is induced from the Killing form on $\fg$.
The quotient space $X=G/K$ is the associated Riemannian symmetric space.
We denote by $d$ the Riemannian distance on $X$ induced by $\|\cdot \|$. We also set $o=[K]\in X$.

 Denote by $w_0\in K$ a representative of the unique element of the Weyl group $\cal W=N_K(A)/M$ such that $\op{Ad}_{w_0}\mathfrak a^+= -\mathfrak a^+$.
  The opposition involution  $\i:\mathfrak a \to \mathfrak a$ is defined by
  $$\i (u)= -\op{Ad}_{w_0} (u) \quad\text{for $u\in \fa$. }$$
   We have $\i(\mu(g))=\mu(g^{-1})$ for all $g\in G$.

  The Furstenberg boundary $\F=G/P$ is isomorphic to $K/M$ as $K$ acts on $\F$ transitively with $K\cap P=M$.
The $\frak a$-valued Busemann function $\beta: \cal F\times G \times G \to\frak a $ is defined as follows: for $\xi\in \cal F$ and $g, h\in G$,
 \be\label{Bu} \beta_\xi ( g, h):=\sigma (g^{-1}, \xi)-\sigma(h^{-1}, \xi)\ee 
where the Iwasawa cocycle $\sigma(g^{-1},\xi)\in \fa$
is defined by the relation $g^{-1}k \in K \exp (\sigma(g^{-1}, \xi)) N$ for $\xi=kP$, $k\in K$.

Let $\Ga<G$ be a Zariski dense discrete subgroup of $G$.
Denote by $\L\subset \fa^+$ the limit cone of $\Gamma$, which is the asymptotic cone of $\mu(\Ga)$, i.e.,
$$\L=\{v\in\frak a^+: v=\lim_{i\to\infty} t_i \mu(\ga_i)\text{ for some }t_i\to 0  \text{ and }\ga_i\to\infty\text{ in } \Ga\}.$$
It is a convex cone with non-empty interior \cite{Ben}.

 The growth indicator function $\psi_{\Gamma}\,:\,\frak a^+ \rightarrow \br \cup\lbrace- \infty\rbrace$  is defined as a homogeneous function, i.e., $\psi_\Gamma (tu)=t\psi_\Gamma (u)$ for all $t\in \br$, such that
  for any unit vector $u\in \frak a^+$,
 \begin{equation} \label{def.GI}
\psi_{\Gamma}(u):=\inf_{{u\in\cal C},{\mathrm{open\;cones\;}\cal C\subset \fa^+}}\tau_{\cal C}
\end{equation}
where $\tau_{\cal C}$ is the abscissa of convergence of the series $\sum_{\ga\in\Ga, \mu(\ga)\in\cal C}e^{-t\norm{\mu(\ga)}}$.
We have $\psi_\Ga \ge 0$ on $\L$ and $\psi_\Ga=-\infty$ outside $\L$.

\subsection*{The generalized BMS-measures $m_{\nu_1,\nu_2}$.}

For $g\in G$, we consider the following visual images:
$$g^+ :=gP\in \F \quad \text{ and}\quad g^- :=gw_0P\in \F.$$ 
Let $\F^{(2)}$ denote the unique open $G$-orbit in $\cal F\times\cal F$ under the diagonal action. In fact, 
$$\F^{(2)}=\{(g^+, g^-): g\in G\}.$$
Then the map $$gM\mapsto (g^+, g^-, b=\beta_{g^-}(e, g))$$ gives a homeomorphism
 $G/M\simeq  \F^{(2)}\times \fa $, called the Hopf parametrization of $G/M$.

For a pair of linear forms $\psi_1, \psi_2\in \mathfrak a^*$ and a pair of  $(\Gamma,\psi_1)$ and $(\Gamma,\psi_2)$ conformal measures $\nu_1$ and $\nu_2$ respectively, define a locally finite Borel measure $\tilde {m}_{\nu_1,\nu_2}$ on $G/M$ 
  as follows: for $g=(g^+, g^-, b)\in \F^{(2)}\times \mathfrak a$,
\begin{equation}\label{eq.BMS0}
d\tilde m_{\nu_1, \nu_2} (g)=e^{\psi_1 (\beta_{g^+}(e, g))+\psi_2( \beta_{g^-} (e, g )) } \;  d\nu_{1} (g^+) d\nu_{2}(g^-) db,
\end{equation}
  where $db=d\ell (b) $ is the Lebesgue measure on $\mathfrak a$.
  By abuse of notation, we also denote by $\tilde m_{\nu_1, \nu_2}$ the $M$-invariant measure on $G$ induced by $\tilde m_{\nu_1, \nu_2}$.
This is always left $\Ga$-invariant  and right $A$ quasi-invariant:
for all $a\in A$,
$$
a_*\tilde m_{\nu_1, \nu_2}=e^{(-\psi_1+\psi_2\circ \i)(\log a)}\,\tilde m_{\nu_1, \nu_2} ;$$
we refer to \cite{ELO} for more details on these measures.
We  denote by
$ m_{\nu_1, \nu_2}$ the $M$-invariant measure on $\Ga\ba G$ induced by $\tilde m_{\nu_1, \nu_2}$.

We will need the following notion:
\begin{dfn}\label{conf}
Let $g_i\in G$ be a sequence whose Cartan decomposition is given by $g_i=k_ia_i\ell_i\in KA^+K$.
As $i\to\infty$,
\end{dfn}
\begin{enumerate}
\item we say that $g_i\to\infty$ regularly if $\alpha(\log a_i)\to\infty$ for all simple root $\alpha$ of $(\fg, \fa)$;
\item we say that $g_i$ converges to $\xi\in\cal F$, if $g_i\to\infty$ regularly and $\lim\limits_{i\to\infty}k_i^+=\xi$;
\item we say that $p_i=g_i(o) \in X$ converges to $\xi\in \cal F$ if $g_i$ does.
\end{enumerate}

We then define the limit set $\La$ of $\Ga$ as the set of all accumulation points of $\Ga(o)$ in $\F$; this is the unique $\Ga$-minimal subset 
(\cite[Lem. 2.13]{LO}, \cite{Ben}).
As in the introduction, we also define the conical limit set:
$$\La_c=\left\{gP\in  \F: 
\begin{array}{c}
\text{there exist $\ga_i\in \Ga$ and $a_i\to \infty$ in $A^+$}       \\
\text{such that
$\ga_i g a_i$ is bounded}
\end{array}\right\}.$$

In the rest of this section, we assume that $\Ga<G$ is a Zariski dense Anosov subgroup
(with respect to $P$) as defined in the introduction. We collect some important properties of Anosov subgroups that we will be using.
\begin{lem} $($\cite{KLP1}, \cite{GW}$)$\label{lem.RA}
If $\Ga<G $ is Anosov, then we have:
\begin{enumerate}
    \item (Regularity) If $\ga_i\to \infty$ in $\Ga$, then $\ga_i\to\infty$ regularly as $i\to\infty$. 
    \item (Antipodality)  If $\xi, \eta\in \La$ are distinct, then $(\xi,\eta)\in\cal F^{(2)}$.
    \item (Conicality) $\La=\La_c$.
\end{enumerate}
\end{lem}
Indeed, these three properties characterize Anosov subgroups
\cite[Thm. 1.1]{KLP1}.
Note that the regularity of (1) implies that
    $\Gamma(o)\cup \La$ is compact. Moreover, 
    by \cite[Lem. 2.10]{LO}, we have:
    \begin{lem}\label{com} 
For any compact subset $Q\subset G$, the union
$\Gamma (Q)\cup \La$ is compact.
    \end{lem}

The following is a consequence of the antipodal property of Anosov subgroups, and plays a key role in this paper.

\begin{lem}\cite[Lem. 8.13]{LO} \label{proper} Let $\Ga<G$ be Anosov. 
For $x=[g]\in \Ga\ba G$, the following are equivalent:
\begin{enumerate}
  \item  $\limsup x A\ne \emptyset$;
 \item  $\limsup xA^+ \cup \limsup xw_0A^+\ne \emptyset $;
 \item $\{gP, g w_0 P\}\cap \La\ne \emptyset.$
\end{enumerate}
\end{lem}

    \begin{thm}\cite{PS}\label{thm.PS}
For $\Ga$ Anosov, we have $$\L\subset \inte \fa^+\cup\{0\}.$$   
    \end{thm}

\begin{cor} \label{inf} If $\Ga<G$ is Anosov and $\text{rank } G\ge 2$, there exists no finite $A$-invariant Borel measure on $\Ga\ba G$.
\end{cor}
\begin{proof} Suppose there exists a finite $A$-invariant Borel measure  $m$
on $\Ga\ba G$. 
Let $v\in \inte\fa^+$. By the Poincare recurrence theorem,
$m$-almost all points are recurrent for the action of $\exp \br v$. In particular, there exist $g\in G$,
$\ga_i\in \Ga$ and $t_i\to +\infty$ such that
$\ga_i g \exp (t_i v)$ is bounded. Then the sequence $\mu(\ga_i^{-1})$ stays in a bounded distance from the ray $\br_+ v$ by \cite[Lem. 4.6]{Ben}; it follows that $v\in \L$.
Therefore $\L=\fa^+\cup\{0\}$. If $\text{rank } G\ge 2$, then
$\fa^+-\{0\}\ne \inte\fa^+$. Hence the claim follows from Theorem \ref{thm.PS}.

\end{proof}

\section{Uniform bound on the multiplicity of shadows}\label{sec.sh}
For
$p\in X=G/K$ and $S>0$,
we set $B(p,S):=\{x\in X: d(x, p)<S\}$.

Recall the notation $o=[K]\in X=G/K$.
For $p \in X$ and $S>0$, the shadow of the ball
$B(p,S)$ as seen from $o$ is defined by
\begin{align*}
O_S(o,p)&:=\{\xi \in \cal F :
\text{ for some $k\in K$ with $\xi=kP$, }kA^+o\cap B(p,S)\neq\emptyset \}.
\end{align*}

\begin{lem} \label{sm}  Let $\Ga<G$ be a Zariski dense Anosov subgroup of $G$. For any $S, D>0$, there exists $q=q(S, D)>0$ such that for any $T>0$, the shadows
 $$\{O_{S}(o,\ga o):T<\norm{\mu(\ga)}<T+D\}$$ have multiplicity at most $q$.
\end{lem}

The rest of this section is devoted to the proof of Lemma \ref{sm}. 

Throughout the section, we fix a compact subset
$Q$ of $G$. 
The notation $x\approx_{Q} y$ means that $x-y$ is contained in a bounded set that depends only on $Q$.
We will simply write $x\approx y$ if the implicit bounded set depends only on $\G$ and $G$.

\begin{lem}\cite[Lem. 4.6]{Ben}\label{lem.lgl}
For all $g\in G$ and $q_1,q_2\in Q$, we have
$$\mu(q_1gq_2)\approx_Q \mu(g).$$
\end{lem}
\begin{lem}\label{lem.q}
Let $a\in A$ and $w\in\cal W$ be such that $w aw^{-1}\in A^+$.
If $Q\cap \ga Q a^{-1}\neq\emptyset$, then $\mu(\ga)\approx_Q \op{Ad}_w\log a$.
\end{lem}
\begin{proof}
If $Q\cap \ga Qa^{-1}\neq\emptyset$, then there exists $q_0, q_0'\in Q$ such that $q_0a=\ga q_0'$.
The conclusion follows from Lemma \ref{lem.lgl}.
\end{proof}

We set $A^-=w_0A^+w_0^{-1}$, and
for any $C>0$, set $A_C:=\{a\in A : \norm{\log a}\leq C\}$.
The following lemma is a key ingredient in the proof of Lemma \ref{sm};
we use the regularity and antipodality of Anosov subgroups.
\begin{lem}\label{lem.RS} Let $\G<G$ be Anosov.
There 
exists
$C_0>1$ depending only on $Q$ such that
whenever  $Q\cap \ga_1 Qa_1^{-1}\cap \ga_2 Qa_2^{-1}\neq\emptyset$ for $\ga_1,\ga_2\in\Ga$ and $a_1,a_2\in A^+$, we have
\begin{enumerate}
    \item
    $a_1^{-1}a_2\in (A^+\cup A^-)A_{C_0}$;

    \item 
    $\mu(\ga_2)\approx_{Q} \mu(\ga_1)+\mu(\ga_1^{-1}\ga_2)$ or $\mu(\ga_1)\approx_{Q} \mu(\ga_2)+\mu(\ga_2^{-1}\ga_1)$.
\end{enumerate}
\end{lem}
\begin{proof} 
We first prove $(1)$. Suppose not.
Then there exists a compact set $Q\subset G$ and sequences
$q_{0,i},q_{1,i},q_{2,i}\in Q$, $a_{1,i}, a_{2,i}\in A^+$ and $\ga_{1,i},\ga_{2,i}\in\Ga$ such that

\begin{align}
    &a_{1,i}^{-1}a_{2,i}\not\in (A^+\cup A^-)A_i,\label{eq.w1}\\
    &q_{0,i}\,a_{1,i}=\ga_{1,i}\,q_{1,i},\quad q_{0,i}\,a_{2,i}=\ga_{2,i}\,q_{2,i}\label{eq.w2}
\end{align} where $A_i=\{a\in A:\|\log a\|\le i\}$.

Observe that \eqref{eq.w1} implies  $a^{-1}_{1,i}a_{2,i}\to\infty$ in $A$ and $a_{1,i}, a_{2,i}\to\infty$ in $A^+$.
Observe that $a_{1,i}, a_{2,i}\to\infty$  regularly, by \eqref{eq.w2} and 
Lemmas \ref{lem.RA} and \ref{lem.lgl}.

Passing to a subsequence, we may assume that for each $m=1,2$,  $q_{m,i}$ converges to some $ q_m\in Q$, and $\ga_{m,i}^{-1}q_{0,i}o$ converges to some element  $\xi\in  \La$ as $i\to\infty$.
Since $\ga_{m,i}^{-1}q_{0,i}o=q_{m,i}a_{m,i}^{-1}o$, it follows that $\xi=q_m^-$ by \cite[Lem. 2.11]{LO} for each $m=1,2$. Therefore $q_m^-\in \La$.
On the other hand, we have
\begin{equation}\label{eq.w3}
\ga_{1,i}^{-1}\ga_{2,i}\,q_{2,i}=q_{1,i}\, a_{1,i}^{-1} a_{2,i}.
\end{equation}
Note that $\ga_{1,i}^{-1}\ga_{2,i}\to\infty$ and there exists $w_i\in\cal W-\{e,w_0\}$ such that $w_i^{-1}a_{1,i}^{-1}a_{2,i} w_i\in A^+$.
Passing to a subsequence, we may assume that $w_i=w$ is constant and $\ga_{1,i}^{-1}\ga_{2,i}\,q_{2,i}\,o$ converges to an element of $\La$ by Lemma \ref{com}.
By \eqref{eq.w3} and \cite[Lem. 2.11]{LO}, it follows that $q_1w^+\in\La$.
This contradicts 
Lemma \ref{lem.RA}, as neither $q_1w^+=q_1^-$ nor $(q_1w^+,q_1^-)\in\cal F^{(2)}$, proving $(1)$.

To prove $(2)$, observe that we have $\mu(\ga_1)\approx_Q \log a_1$, $\mu(\ga_2)\approx_Q \log a_2$ by Lemma \ref{lem.q}, since $Q\cap \ga_1 Qa_1^{-1}\cap \ga_2Qa_2^{-1}\neq\emptyset$. On the other hand, it follows from $(1)$ that 
$$\mu(\ga_2^{-1}\ga_1)\approx_{Q} \log a_1^{-1}a_2\quad\text{or}\quad \mu(\ga_1^{-1}\ga_2)\approx_{Q} \log a_2^{-1}a_1.$$
Hence $(2)$ is proved.
\end{proof}

The following lemma follows from  Theorem \ref{thm.PS} and the fact that the angle between two walls of a  Weyl chamber is at most $\pi/2$.
\begin{lem}\label{lem.par}
There exist constants $\beta_1,\beta_2>0$ depending only on $\Ga$ such that
for all $x,y\in\mu(\Ga)$, we have $$
\norm{x+y}^2\geq \norm{x}^2+\norm{y}^2+\beta_1\norm{x}\norm{y}-\beta_2.
$$
\end{lem}

\subsubsection*{Proof of Lemma \ref{sm} }
Suppose that there exist $\xi\in \bigcap_{i=1}^n O_{S}(o,\ga_io)$ and  $T<\norm{\mu(\ga_i)}<T+D$ for some $\ga_i$ $(i=1,\cdots,n)$.
Set $Q:=KA_{S}^+K$. Choose $k\in K$ such that $\xi=kP$.
Then  $d(kA^+o, \ga_i o)\le S$. It follows that there exists  a sequence $a_1,\cdots, a_n\in A^+$ such that
 $k\in Q\cap\ga_1Qa_1^{-1}\cap\cdots\cap \ga_nQa_n^{-1}$.

We claim that there exists $D'=D'(Q, D)>0$ such that  \begin{equation}\label{eq.D}\max_{i,j} \|\mu(\ga_i^{-1}\ga_j)\| <D'.\end{equation}
This implies that $n\le \#\{\ga\in \G:\|\mu(\ga)\|\le D'\}$.

To prove \eqref{eq.D}, we apply Lemma \ref{lem.RS}(2) to each pair $(\ga_i,\ga_j)$; suppose first that $\mu(\ga_j)\approx_Q \mu(\ga_i)+\mu(\ga_i^{-1}\ga_j)$.
Since $
\|\mu(\ga_j)\|
\le T+D$, there exists $D_1=D_1(Q)>0$ such that
\begin{equation}\label{eq.T}
  \norm{\mu(\ga_i)+\mu(\ga_i^{-1}\ga_j)}^2\le  (\|\mu(\ga_j)\|+ D_1)^2\le 
  (T+D+D_1)^2.
\end{equation}
Set $D_2=D+D_1$.
By Lemma \ref{lem.par} and  \eqref{eq.T}, we deduce that
$$
\beta_1\norm{\mu(\ga_i^{-1}\ga_j)}T+\norm{\mu(\ga_i^{-1}\ga_j)}^2<2D_2T+D_2^2+\beta_2,
$$
in particular, $\norm{\mu(\ga_i^{-1}\ga_j)}<\max (\sqrt{D_2^2+\beta_2},2D_2\beta_1^{-1})$.
The other case of Lemma \ref{lem.RS}(2) also yields the same conclusion by a symmetric argument.

This proves the claim \eqref{eq.D}.\qed

We remark that the boundedness of the multiplicity of the {\it intersection } of shadows and the limit set for projective Anosov representations, with respect to the {\it word length} $|\ga|$ is given \cite[Prop. 3.5]{PSW}.

\section{Poincare series and the average of correlations}
Let $\G<G$ be a Zariski dense Anosov subgroup. We fix $\psi\in \fa^*$  and  a $(\Ga, \psi)$-conformal measure $\la_\psi$ on $\cal F$ (not necessarily supported on $\La$).  
We assume that 
$$\sum_{\ga\in \Ga}e^{-\psi(\mu(\ga))} =\infty.$$ 
This implies that $\psi$ is $\G$-critical
by \cite[Lem. III.1.3]{Q4}.
Therefore, there exists a
$(\Ga,\psi\circ\i)$-conformal measure, say $\la_{\psi\circ\i}$, e.g., as constructed by Quint.

Let 
$$\tilde{\mathsf m}=\tilde m_{\la_\psi,\la_{\psi\circ \i}}^{\BMS}$$ denote the generalized $\BMS$ measure on $G$, which is left $\Ga$-invariant and right $AM$-invariant.

The notations $x\lesssim_z y$ (resp. $x\ll_z y$) are to be understood that $x\leq y+C$ (resp. $x\leq Cy$) for some constant $C>0$ that depends on $z$.

\medskip 
The main aim of this section is to prove the following proposition.
For $r>0$ and any subset $S\subset A$, we set $S_r=\{a\in S: 
\|\log a\|
\le r\}$.

\begin{prop}\label{p1}  
Let $Q_r=KA_r^+ K A_r$ for $r>0$.
For any sufficiently large $r>1 $, the following holds:
for any $T\ge 1$,
$$
\int_{A_T^+}\int_{A_T^+}\sum_{\ga_1,\ga_2\in\Ga} \tilde{\mathsf m}(Q_r\cap \ga_1 Q_r a_1^{-1}\cap \ga_2 Q_r a_2^{-1})\,da_1\,da_2\ll \bigg(\sum_{
    \substack{
    \ga\in\Ga,\\
    \norm{\mu(\ga)}\leq T
    }
    }e^{-\psi(\mu(\ga))}\bigg)^2
$$
and 
$$
\int_{A_T^+}\sum_{\ga\in\Ga} \tilde{\mathsf m}(Q_r\cap \ga Q_r a^{-1})\,da\gg \sum_{
    \substack{
    \ga\in\Ga,\\
    \norm{\mu(\ga)}\leq T
    }
    }e^{-\psi(\mu(\ga))}
$$ where the implied constants depend only on $r$.
\end{prop}

The rest of this section is devoted to the proof of this proposition, given as the proofs of Propositions \ref{prop.sub1} and \ref{prop.sub2}.

The $\frak a$-valued Gromov product on $\cal F^{(2)}$ is defined as follows: 
for
$(g^+, g^-) \in\cal F^{(2)}$,
$$
\cal G(g^+,g^-): = \beta_{g^+}(e,g)+\op i{\left(\beta_{g^-}(e,g)\right)};
$$ this is well-defined independent of the choice of $g\in G$.
\begin{lem}\cite[Prop. 8.12]{BPS}\label{lem.GP}
There exist $c$, $c'>0$ such that for all $g\in G$,
$$
c^{-1}\norm{\cal G(g^+,g^-)}\leq d(o, gAo)\leq c\norm{\cal G(g^+,g^-)}+c'.
$$
\end{lem}

\medskip

For $r>0$, let 
$$G_r=KA_r^+K$$ and
$$\cal L_r(o, go):=\{(h^+,h^-)\in\cal F^{(2)}: h\in G_r,\;
ha\in   gG_r\text{ for some $a\in A^+$}\}.$$
\begin{lem}\label{lem.u1}
For any $g\in G$ and $r>0$, we have
$$
\cal L_r(o,go)\subset O_{2r}(o,go)\times O_{2r}(go,o).
$$
\end{lem}
\begin{proof}
Let $(\xi,\eta)\in \cal L_r(o,go)$.
Then there exists $h\in G_r$ such that $(h^+,h^-)=(\xi,\eta)$ and $d(hA^+o,go)<r$.
Let $k\in K$ be such that $k^+=h^+$.
Then the Hausdorff distance between $kA^+o$ and $hA^+o$ is given by $d(o,ho)<r$
\cite[1.6.6(4)]{Eb}
and hence $d(kA^+o, go)<2r$.
It follows that $\xi=k^+\in O_{2r}(o,go)$.
A similar computation shows that $\eta\in O_{2r}(go,o)$.
\end{proof}
\begin{lem}\label{lem.q2} Let $r>0$.
If $g\in Q_r\cap \ga Q_ra^{-1}$ for $\ga\in\Ga$ and $a\in A^+$, then
\begin{enumerate}
    \item 
    $(g^+,g^-)\in\cal L_{2r}(o,\ga o)$.
    \item
    $|\psi(\cal G(g^+,g^-))|<2\norm{\psi}cr$ where $c$ is from Lemma \ref{lem.GP}.
    \item
    $gA\cap Q_r\cap \ga Q_ra^{-1}\subset gA_{4r}$.
\end{enumerate}
\end{lem}
\begin{proof}
$(1)$ follows from the definition since $Q_r\subset G_{2r}$.
$(2)$ follows from Lemma \ref{lem.GP} and the fact that $d(gAo,o)\leq d(go,o)<2r$.
$(3)$ follows from the stronger inclusion $gA\cap Q_r\subset gA_{4r}$; if $g$, $gb\in Q_r$ for some $b\in A$, then $b\in Q_r\cdot Q_r\subset G_{4r}$ since $Q_r\subset G_{2r}$.
Note that $G_{4r}\cap A=A_{4r}$.
\end{proof}

We will need the following shadow lemma:
\begin{lem}\label{lem.shadow}  \cite[Lem. 7.8]{LO}: 
 There exists $S_0>0$ such that for all $S\geq S_0$ and all $\ga\in\Ga$, we have
$$
 e^{-\psi(\mu(\ga))}\ll \la_\psi(O_S(o,\ga o))\ll e^{-\psi(\mu(\ga))}.
$$ with implied constants independent of $\ga$.
\end{lem}

\begin{lem}\label{lem.one}
Let $r>0$. For any $a\in A^+$, we have
$$
\tilde{\mathsf m}(Q_r\cap \ga Q_r a^{-1})\ll_{\,r} e^{-\psi(\mu(\ga))}.
$$
\end{lem}
\begin{proof}
By Lemmas \ref{lem.u1}, \ref{lem.q2} and \ref{lem.shadow}, we have
\begin{align*}
    &\tilde{\mathsf m}(Q_r\cap \ga Q_r a^{-1})\\
    &=\int_{\cal L_{2r}(o,\ga o)}\left(\int_{A}\mathbbm{1}_{Q_r\cap\ga Q_ra^{-1}}(gb)e^{\psi(\cal G(g^+,g^-))}\,db\right)\,d\la_{\psi}(g^+)\,d\la_{\psi\circ\i}(g^-)\\
    &\leq \int_{O_{4r}(o,\ga o)\times O_{4r}(\ga o,o)\cap\cal F^{(2)}}\op{Vol}(A_{4r})e^{2\norm{\psi}cr}\,d\la_{\psi}(g^+)\,d\la_{\psi\circ\i}(g^-) \\&
    \ll_{\,r} e^{-\psi(\mu(\ga))},
\end{align*}
which proves the lemma.
\end{proof}
The following is easy to prove (cf. \cite[Lem. 5.14]{BLLO}).
\begin{lem}\label{lem.S}
There exists $\ell_0>0$ such that any $\ga\in\Ga$ with $\norm{\mu(\ga)}>\ell_0$ and any $(\xi,\eta)\in O_{S_0}(o,\ga o)\times O_{S_0}(\ga o,o)$ satisfies $\norm{\cal G(\xi,\eta)}<\ell_0$.
\end{lem}

 In the rest of this section, we fix constants
 $S_0$, $\ell_0$ $c$, $c'$  from Lemmas  \ref{lem.GP}, \ref{lem.shadow} and \ref{lem.S}.  
  We set
\be \label{ro} r_0:=S_0+c\ell_0+c'+1 .\ee 

\begin{lem}\label{lem.lb1}
For all $r>r_0$, there exists $C_2=C_2(r)>0$ such that for any $T\ge C_2$
 and any $g\in G$ with 
 $$(g^+,g^-)\in \bigcup\{O_{S_0}(o,\ga o)\times O_{S_0}(\ga o,o): \ga\in \Ga,  \ell_0<\norm{\mu(\ga)}<T-C_2\},$$
we have  $$\int_{A_T^+} \int_A\mathbbm{1}_{Q_r\cap \ga Q_ra^{-1}}(gb)\,db\,da\geq \op{Vol}(A_r)\op{Vol}(A_1^+).
$$
\end{lem}
\begin{proof}
Let $C_2'=C_2'(r)$ be the implied constant in Lemma \ref{lem.q} associated to $Q=Q_r$.
Set $C_2:=C_2'+1$. Let $T>C_2$.
Let $g\in G$ and $\ga\in\Ga$ be such that $\ell_0<\norm{\mu(\ga)}<T-C_2$
 and $(g^+,g^-)\in O_{S_0}(o,\ga o)\times O_{S_0}(\ga o,o)$.
 By Lemmas \ref{lem.GP} and \ref{lem.S},
 we have $d(o, gAo)\le c \|\mathcal G (g^+, g^-)\| +c' \le c\ell_0+c'.$
 Therefore, we may assume without loss of generality that
$d(o,go)\le c\ell_0+c'$ by replacing $g$ by an element of $gA$.

Since $g^+\in O_{S_0}(o,\ga o)$, there exists $k\in K$ such that $k^+=g^+$ and $d(kao,\ga o)<{S_0}$ for some $a\in A^+$.

Since $d(kao, gao)\le d(o, go)$ by \cite[1.6.6(4)]{Eb},
we get \begin{align*} d(\ga o, gA^+o) & \le d (\ga o, ka o) + d(ka o, gao)
\\ &\le  d (\ga o, ka o) + d(k o, go) \\& \le  S_0+ c\ell_0 +c'=
r_0-1 .\end{align*}
Since $r\ge r_0$,
we have 
 $g\in G_{r-1}$ and  $ga_0\in \ga G_{r-1}$ for some $a_0\in A^+$.

Therefore $g\in G_{r-1}\cap \ga
G_{r-1}a_0^{-1}$.
 By Lemma \ref{lem.q}, this implies that
 $\norm{\mu(\ga)-\log a_0}\le C_2'$. Since $\|\mu(\ga)\|\le T-C_2'-1$,
 we have $a_0\in A_{T-1}^+$, and hence $a_0A_1^+\subset A_T^+$. Since $Q_r=G_rA_r$, we have
$gA_r\subset Q_r\cap \ga Q_ra^{-1}$  for all $a\in a_0A_1$.
Therefore
\begin{align*} & \int_{A_T^+} \int_A\mathbbm{1}_{Q_r\cap \ga Q_ra^{-1}}(gb)\,db\,da\geq
\int_{a_0A_1^+}\int_{A}\mathbbm{1}_{Q_r\cap \ga Q_r a^{-1}}(gb) \,db\,da
\\ & \ge \int_{a_0A_1^+}\int_{A}\mathbbm{1}_{gA_r}(gb) \,db\,da
\geq \op{Vol}(A_1^+)\op{Vol}(A_r).
\end{align*}
This finishes the proof.
\end{proof}

We now deduce the following  from 
 Lemma \ref{sm} and the shadow lemma \ref{lem.shadow}.
\begin{lem}\label{lem.mult}
For any $D>0$, we have:
$$
\sup_{T>0}\sum_{
\substack{
\ga\in\Ga,\\
T<\norm{\mu(\ga)}<T+D
}
}e^{-\psi(\mu(\ga))}<\infty.
$$
\end{lem}
\begin{proof} 
For any $T>0$,
$$
\sum_{T<\norm{\mu(\ga)}<T+D}e^{-\psi(\mu(\ga))}\ll \sum_{T<\norm{\mu(\ga)}<T+D} \la_\psi(O_{S_0}(o,\ga o))\le 
q(S_0,D)
$$ where 
$q(S_0,D)$ is
given by Lemma \ref{sm}. This proves Lemma \ref{lem.mult}.
\end{proof}

We are now ready to give estimates for correlation functions in terms of Poincar\'e series, which was the main goal of the section.

\begin{prop}\label{prop.sub1}
For all $r>r_0$, we have, for all $T\ge 1$,
$$
\int_{A_T^+}\int_{A_T^+}\sum_{\ga_1,\ga_2\in\Ga} \tilde{\mathsf m}(Q_r\cap \ga_1 Q_r a_1^{-1}\cap \ga_2 Q_r a_2^{-1})\,da_1\,da_2\ll_{r} \bigg(\sum_{
    \substack{
    \ga\in\Ga,\\
    \norm{\mu(\ga)}\leq T
    }
    }e^{-\psi(\mu(\ga))}\bigg)^2.
$$
\end{prop}
\begin{proof}
Let $C_0>0$ be as in Lemma \ref{lem.RS}(1) associated to $Q=Q_r$.
Set 
\begin{equation}\label{eq.i1}
E_{\ga_1,\ga_2}:=\left\{(a_1,a_2)\in A_T^+\times A_T^+:
\begin{array}{c}
     Q_r\cap \ga_1 Q_r a_1^{-1}\cap \ga_2 Q_r a_2^{-1}\neq\emptyset\\
     \mu(\ga_2)\approx_{Q_r} \mu(\ga_1)+\mu(\ga_1^{-1}\ga_2)
\end{array}\right\},
\end{equation}
where the implied constant for $\approx_{Q_r}$ is chosen to be the one
 in Lemma \ref{lem.RS}(2) with $Q=Q_r$.
Note that by Lemma \ref{lem.q}, the subset
$E_{\ga_1,\ga_2}$ is contained in some bounded ball around
$(\mu(\ga_1),\mu(\ga_2))$ whose radius depends only on $r$.  Hence
the volume of $E_{\ga_1,\ga_2}$ has a uniform upper bound depending only on $r$. 
Observe that if there exists $(a_1, a_2)\in E_{\ga_1, \ga_2}$,
then $\norm{\mu(\ga_i)}\approx \|\log a\|   \le T$. 
Since the angle between any two walls of $\fa^+$ is at most $\pi/2$, we deuce
$\norm{\mu(\ga_1^{-1}\ga_2)} \le \|\mu(\ga_1) +\mu(\ga_1^{-1}\ga_2)\|
\lesssim_{\,r} \|\mu(\ga_2)\|  \lesssim  T$.

Therefore we get
\begin{align}\label{eq.ub1}
    &\int_{A_T^+}\int_{A_T^+}\sum_{\ga_1,\ga_2\in\Ga} \tilde{\mathsf m}(Q_r\cap \ga_1 Q_r a_1^{-1}\cap \ga_2 Q_r a_2^{-1})\,da_1\,da_2
    \notag\\
    &\leq 2 \int_{A_T^+}\int_{A_T^+}\sum_{\ga_1,\ga_2\in\Ga} \tilde{\mathsf m}(Q_r\cap \ga_1 Q_r a_1^{-1}\cap \ga_2 Q_r a_2^{-1})\mathbbm{1}_{E_{\ga_1,\ga_2}}(a_1,a_2)\,da_1\,da_2
    \notag\\
    &\ll_{\,r} \int_{A_T^+}\int_{A_T^+} \sum_{\ga_1,\ga_2\in\Ga} e^{-\psi(\mu(\ga_2))}\mathbbm{1}_{E_{\ga_1,\ga_2}}(a_1,a_2)\,da_1\,da_2
    \notag\\
    &\ll_{\,r}\sum_{
    \substack{
    \ga_1,\ga_2\in\Ga,\\
    \norm{\mu(\ga_1)} \lesssim_{\,r} T,\\
    \norm{\mu(\ga_1^{-1}\ga_2)} \lesssim_{\,r} T}
    }
    e^{-\psi(\mu(\ga_1))}e^{-\psi(\mu(\ga_1^{-1}\ga_2))}\int_{A_T^+}\int_{A_T^+} \mathbbm{1}_{E_{\ga_1,\ga_2}}(a_1,a_2)\,da_1\,da_2
    \notag\\
    &\ll_{\,r} \bigg(\sum_{
    \substack{
    \ga\in\Ga,\\
    \norm{\mu(\ga)}\lesssim_{\,r}T}
    }
    e^{-\psi(\mu(\ga))}\bigg)^2;
\end{align}
note here that the first inequality follows from Lemma \ref{lem.RS}(2) and the symmetricity of the expression with respect to $\ga_1,\ga_2$.
The second inequality is due to Lemma \ref{lem.one}.
The third inequality is valid again by Lemma \ref{lem.RS}(2).
The last inequality is obtained by reindexing $\ga_1^{-1}\ga_2\in\Ga$ with a new variable.
Finally note that \eqref{eq.ub1} together with Lemma \ref{lem.mult} finishes the proof.
\end{proof}

\begin{prop}\label{prop.sub2}
For all $r>r_0$, we have, for all $T>0$,
$$
\int_{A_T^+}\sum_{\ga\in\Ga} \tilde{\mathsf m}(Q_r\cap \ga Q_r a^{-1})\,da\gg_{r} \sum_{
    \substack{
    \ga\in\Ga,\\
    \norm{\mu(\ga)}\leq T
    }
    }e^{-\psi(\mu(\ga))}.
$$
\end{prop}
\begin{proof}
By Lemma \ref{lem.S}, for all $a\in A^+$ and $\ga\in\Ga$ with $\norm{\mu(\ga)}\geq\ell_0$,
\begin{align*}
    &\tilde{\mathsf m}(Q_r\cap \ga Q_r a^{-1})\\
    &\gg\int_{O_{S_0}(o,\ga o)\times O_{S_0}(\ga o,o)\cap\cal F^{(2)}}\int_A \mathbbm{1}_{Q_r\cap \ga Q_ra^{-1}}(gb) \,db\,d\la_\psi(g^+)\,d\la_{\psi\circ\i}(g^-).
\end{align*}
Hence by Lemmas \ref{lem.shadow} and \ref{lem.lb1},
$$
\int_{A_T^+}\sum_{\ga\in\Ga}\tilde{\mathsf m}(Q_r\cap \ga Q_r a^{-1})\,da 
    \gg_{r} \sum_{
    \substack{
    \ga\in\Ga\\
    \norm{\mu(\ga)}\leq T-C_2
    }
    }e^{-\psi(\mu(\ga))}.
$$
This finishes the proof in view of Lemma \ref{lem.mult}.
\end{proof}

\section{Conical limit points and Poincare series}

We begin by recalling:
\begin{lem} \label{divv} \cite[Lem. 7.11]{LO}.
Let $\G<G$ be a Zariski dense discrete subgroup and $\psi\in \fa^*$.
If there exists a $(\Ga,\psi)$-conformal measure $\la_\psi$ with
$\la_\psi(\La_c)>0$, then 
$$\sum_{\ga\in \Ga}e^{-\psi(\mu(\ga))} =\infty.$$
\end{lem}

The goal of this section is to establish the converse 
for Anosov subgroups:
\begin{prop}\label{div} Let $\Ga<G$ be a Zariski dense Anosov subgroup of $G$.
Let $\psi\in \fa^*$.
If $\sum_{\ga\in \Ga}e^{-\psi(\mu(\ga))} =\infty$, then
for any $(\Ga, \psi)$-conformal measure $\la_\psi$,
we have $$\la_\psi(\La_c)>0.$$
\end{prop}
We will need the following version of the Borel-Cantelli lemma:
\begin{lem}\label{lem.AS} Let 
$(\Omega, \mathsf M)$
be a Borel probability  measure space.
Let $\mathbbm{1}_{P_a}(\om)$ be a jointly measurable function of $(a,\om)\in A^+\times\Om$.
Suppose that  $\int_{A^+}\mathsf M(P_a)\,da=\infty$ and
\begin{equation}\label{eq.inf}
\liminf_{T\to\infty}\frac{\int_{A_T^+}\int_{A_T^+}\mathsf M(P_a\cap P_b)\,db\,da}{\left(\int_{A_T^+}\mathsf M(P_a)\,da\right)^2}\leq C 
\end{equation}
for some $C<\infty$.
Then we have
$$
\mathsf M(\{\om:\int_{A^+}\mathbbm{1}_{P_a}(\om)\,da=\infty\})>\frac{1}{C}.
$$
\end{lem}

\begin{proof} The proof is an easy adaption of \cite{AS}.
Set 
$$
b(t,\om)=\frac{\int_{A_t^+}\mathbbm{1}_{P_a}(\om)\,da}{\int_{A_t^+}\mathsf  M(P_a)\,da}\;\; \text{ and }\;\; \Om_0=\{\om:\int_{A^+}\mathbbm{1}_{P_a}(\om)\,da=\infty\}.
$$
Since $\int_{A^+}\mathsf M(P_a)\,da=\infty$, we have $b(t,\om)\to 0$ as $t\to\infty$ for all $\om\in\Om_0^c$.
On the other hand, by \eqref{eq.inf}, there exists $t_n\to\infty$ such that 
$$
C_n:=\int_\Om b(t_n,\om)^2\,d\mathsf M(\om)\to C_\infty
$$
for some $C_\infty\leq C$, as $n\to\infty$.
Since the family of functions $\om\mapsto b(t_n,\om)$ is uniformly bounded in their $L^2$-norms, they are uniformly integrable.
Hence $\int_{\Om_0^c}b(t_n,\om)\,d\mathsf M(\om)\to 0$ as $n\to\infty$.
Since $\int_\Om b(t,\om)\,d\mathsf M(\om)=1$, it follows that $$\lim_{n\to \infty} \int_{\Om_0}b(t_n,\om)\,d\mathsf M(\om)= 1.$$
By the Cauchy-Schwartz inequality, for all $n\ge 1$,
$$
\mathsf M(\Om_0)C_n=\int_{\Om_0} 1^2\,d\mathsf M(\om) \int_{\Om_0}b(t_n,\om)^2\,d\mathsf M(\om)\geq \left(\int_{\Om_0}b(t_n,\om)\,d\mathsf M(\om)\right)^2.
$$
Therefore 
$$\mathsf M(\Omega_0) \ge \lim_{n\to \infty}\frac{1}{ C_n}
=\frac{1}{C_\infty} \ge \frac{1}{C}$$ and the conclusion follows.
\end{proof}

Let $\mathsf m$
denote the measure on $\Ga\ba G$ induced from $\tilde{\mathsf m}=\tilde m_{\la_\psi,\la_{\psi\circ\i}}^{\BMS}$.
For $r>0$, set $G_r:=KA_r^+K$, $Q_r:=G_rA_r$ and $\mathsf M:= \mathsf m|_{\Ga Q_r}$.
For all $a\in A^+$, set
$$
P_a:=\Ga(Q_r\cap \Ga Q_ra^{-1})\subset\Ga\ba G.
$$
We will prove:
\begin{prop}\label{prop.main}
For all sufficiently large $r>1$, we have, for all $T\ge 1$,
\begin{align*}
    &\int_{A_T^+}\mathsf M(P_a)\,da \gg_r \sum_{
    \substack{
    \ga\in\Ga,\\
    \norm{\mu(\ga)}\leq T
    }
    }e^{-\psi(\mu(\ga))}\text{ and }\\
    &\int_{A_T^+}\int_{A_T^+}\mathsf  M(P_{a_1}\cap P_{a_2})\,da_1\,da_2\ll_r \bigg(\sum_{
    \substack{
    \ga\in\Ga,\\
    \norm{\mu(\ga)}\leq T
    }
    }e^{-\psi(\mu(\ga))}\bigg)^2.
\end{align*}
\end{prop}
\begin{proof}
Note that for all $a,a_1,a_2\in A^+$,
\begin{align*}
    \mathsf M(P_a)&\gg_r \sum_{\ga\in\Ga}\tilde{\mathsf m}(Q_r\cap\ga Q_ra^{-1})\text{ and}\\
    \mathsf M(P_{a_1}\cap P_{a_2})&\ll_r \sum_{\ga_1,\ga_2\in\Ga}\tilde{\mathsf m}(Q_r\cap\ga_1 Q_ra_1^{-1}\cap\ga_2 Q_ra_2^{-1});
\end{align*}
Indeed, if $Q_r$ is small enough so that it injects to $\Gamma\ba G$, then we have equalities in the above  by the definition of measures. In general, the above inequalities  follow by covering $Q_r$ by finitely many subsets which inject to $\Gamma\ba G$ and the implied constants depend only on
the multiplicity of the covering.
With this observation, the proposition follows from Propositions \ref{prop.sub1} and \ref{prop.sub2}.
\end{proof}

\subsection*{Proof of Proposition \ref{div}.}
By  Proposition \ref{prop.main} and Lemma \ref{lem.AS}, 
the following set has positive $\mathsf m$-measure:
\begin{equation}\label{eq.E}
W_r:=\{[g]\in \Ga Q_r: \int_{A^+}\mathbbm{1}_{\Ga Q_r}(ga)\,da=\infty\}
\end{equation}
for all $r$ large enough.
On the other hand, note that for all $[g]\in W_r$, 
there exists 
$a_i\to\infty$ in $A^+$
such that $[g]a_i$ is bounded; and hence $g^+\in \La_c$.
Therefore $\la_\psi(\La_c)>0.$ This finishes the proof.
$\qed$

\section{Dichotomy theorem for the $A$-action}\label{secc}
We begin by recalling the notion of complete conservativity and dissipativity. Let $H$ be either a countable group or
a connected  closed
subgroup of $A$. We denote by $dh$ the Haar measure on $H$.
Consider the dynamical system $(\Omega, \mu, H)$
where $\Omega$ is a separable, locally compact and $\sigma$-compact topological space on which $H$ acts continuously and $\mu$ is a Radon measure which is quasi-invariant by $H$. 

A Borel subset $B\subset \Omega$ is called wandering
if $\int_H {\mathbbm 1}_B(h. w)dh <\infty$ for $\mu$-almost all $w\in B$. The Hopf decomposition theorem says that $\Omega$ can be written as the disjoint union $\Omega_C\cup \Omega_D$ of $H$-invariant subsets where $\Omega_D$ is a countable union of wandering subsets which is maximal in the sense that $\Omega_C$
does not contain any wandering subset of positive measure. If $\mu(\Omega_D)=0$ (resp. $\mu(\Omega_C)=0$), the system is called completely conservative (resp. dissipative). When 
$(\Omega, \mu, H)$ is ergodic, $H$ is countable and $\mu$ is atom-less, then it is completely conservative (cf. \cite[Thm. 14]{Ka}).

The following is standard (cf. \cite[Proof of Thm . 4.2]{BLLO})
\begin{lem} \label{ccc} Suppose that $\mu$ is $H$-invariant.
Then  $(\Omega, H, \mu)$ is completely conservative if and only if for $\mu$ a.e. $x\in \Omega$,
there exists a compact subset $B_x\subset \Omega$ such that $\int_{h\in H} {\mathbbm 1}_{B_x}(h.x)\; dh =\infty$. 
\end{lem}
 \begin{proof} 
 Suppose $(\Omega, H, \mu)$ is completely conservative.
 Suppose that there exists a $\mu$-positive measurable subset $E\subset \Omega$ such that for all $x\in E$, $\int_{h\in H} {\mathbbm 1}_{B}(h.x)\; dh <\infty$ for any compact subset $B$. Then any compact subset of $E$ with positive measure is a wandering set.  This proves the only if direction.
Now suppose that for $\mu$ a.e. $x\in \Omega$,
there exists a compact subset $B_x\subset \Omega$ such that $\int_{h\in H} {\mathbbm 1}_{B_x}(h.x)\; dh =\infty$.
Assume that there exists a wandering set $W \subset \Omega$  with
$0<\mu (W)<\infty$. 
By the $\sigma$-compactness of $\Omega$, there exists
a compact subset $B\subset \Omega$
such that \be\label{mw} \mu \{x\in W:\int_{H} \mathbbm 1_B(h.x) dh=\infty\}\ge \mu (W)/2.\ee

For any $n\in \mathbb N$, set
$ W_n:=\left\{w \in W: \int_{H} \mathbbm{1}_{W}(h.w)\,dh\le n\right\}$.
Fix $n$ such that $\mu(W_n)>\mu(W)/2$.
 For any
compact subset $C\subset H$, we get, using the $H$-invariance of $\mu$,
\begin{align*}
    &\int_{W_n} \int_{C} \mathbbm{1}_B (h.w)\,dh \,d\mu = \int_{C} \int_{W_n} \mathbbm{1}_B (h.w)\,d\mu\, dh \\& 
    =\int_{C} \mu(B \cap h W_n) \,dh =\int_{C} \int_{B \cap HW_n} \mathbbm{1}_{W_n} (h^{-1}.x) d\mu dh \\& = \int_{B \cap HW_n} \int_{C} \mathbbm{1}_{W_n} (h^{-1}.x)\,dh\, d\mu  \le \int_{B \cap H W_n} \int_{H} \mathbbm{1}_{W_n} (h^{-1}.x)\,dh\, d\mu \\ &\le \int_{B \cap HW_n} n \;d\mu\le  n \cdot \mu(B ) < \infty.
\end{align*}

Hence  $\int_{W_n} \int_{H} \mathbbm{1}_B (h.w)\,dh\, d\mu <\infty$; so
$$\mu \{ x\in W: \int_{H} \mathbbm{1}_B (h.w)\,dh  <\infty\}\geq \mu (W_n)>\mu(W)/2,$$
contradicting \eqref{mw}.
\end{proof}

\subsection*{Proof of Theorem \ref{mp}}
The equivalence $(1)\Leftrightarrow (2)$ follows from  \cite[Prop. 2.8, Lem. 4.4]{Q2} and \cite[Prop. 2.10]{PS} (see also \cite[Cor. 7.12]{LO}).

The equivalence $(3)\Leftrightarrow (4)$ follows because the restriction of
$\la_\psi$ to any $\G$-invariant measurable subset is again a $(\G,\psi)$-conformal measure, up to a positive constant multiple, if not-trivial.

The equivalence $(5)\Leftrightarrow (6)$  follows from
the $\G$-equivariant homeomorphism  $\F^{(2)}\simeq G/AM$ and Lemma \ref{ccc}.
More precisely, for any $\Ga$-invariant subset $Z\subset\cal F^{(2)}$, define a $\Ga$-invariant subset $\tilde Z\subset G/M$ by
$$
\tilde Z:=Z\times A\subset\cal F^{(2)}\times A
$$
using the Hopf parametrization $\F^{(2)}\times A\simeq G/M$.
We may view $\tilde Z$ as an $A$-invariant subset of $\Ga\ba G/M$ as well.
It follows from Lemma \ref{ccc} that the assignment $Z\mapsto \tilde Z$ preserves the conservativity (and complete dissipativity) of the action of $\Ga$ on $(\cal F^{(2)},\lambda_\psi\otimes \lambda_{\psi\circ\i}|_{\cal F^{(2)}})$ and the action of $A$ on $(\Ga\ba G/M,m_{\la_\psi, \la_{\psi \circ \i}})$.
The equivalence $(5)\Leftrightarrow (6)$ now follows in view of the Hopf decompositions (see the beginning of Section 6) for the systems $(\cal F^{(2)},\lambda_\psi\otimes \lambda_{\psi\circ\i}|_{\cal F^{(2)}},\Ga)$ and $(\Ga\ba G/M,m_{\la_\psi, \la_{\psi \circ \i}},A)$.

The direction $(3)\Rightarrow (1)$ is proved in \cite[Lem. 7.11]{LO}.

The direction $(1)\Rightarrow (3)$ was shown in Proposition \ref{div}.

For the implication $(4)\Rightarrow (5)$,
we will use that all $(1)-(4)$ are equivalent.
Suppose that $\lambda_\psi(\La_c)=1$, and hence $\psi$ is $\G$-critical.
In this case, see (\cite{Sam}, \cite[Cor. 4.9]{LO}) for the  $AM$-ergodicity of $ m_{\la_\psi, \la_{\psi \circ \i}}$.
Hence $\la_\psi\otimes\la_{\psi\circ\i}|_{\cal F^{(2)}}$ is ergodic.
To prove it is conservative, observe that 
since $\lambda_\psi(\La_c)=1$, 
and  no point in $\La_c$ can be an atom  by Lemma \ref{lem.shadow},
 $\lambda_\psi$ is atom-less.
Therefore 
 $\la_\psi\otimes\la_{\psi\circ\i}|_{\cal F^{(2)}}$ has no atom. This implies
$\la_\psi\otimes\la_{\psi\circ\i}|_{\cal F^{(2)}}$ is conservative by \cite[Thm. 14]{Ka}.  Therefore $(5)$ follows.
  To show $(5)\Rightarrow(4)$,
  suppose that
 $m_{\la_\psi, \la_{\psi \circ \i}}$ is completely conservative and ergodic. 
Fix any compact subset $B\subset \Ga\ba G/M$.
Then by Lemma \ref{proper}, we have for $g\in G$, 
\begin{equation*} \label{eqq} \limsup \Gamma g A M\cap B\ne \emptyset\text{ if and only
if }\limsup \Gamma g (A^+\cup w_0A^+w_0^{-1}) M\cap B\ne \emptyset .\end{equation*}
Therefore, it follows that 
$\max (\la_\psi (\La_c), \la_{\psi\circ \i}(\La_c))>0$.
On the other hand,  Since $\sum_{\ga\in \Ga}e^{-\psi(\mu(\ga))} = \sum_{\ga\in \Ga}e^{-\psi\circ \i(\mu(\ga))} $, the equivalence $(4)\Leftrightarrow (1)$ implies that 
$\min (\la_\psi (\La_c), \la_{\psi\circ \i}(\La_c))>0$.
This proves (4).

If $\lambda_\psi(\La_c)=0$, 
the measure $\la_\psi\otimes\la_{\psi\circ\i}|_{\cal F^{(2)}}$ must be non-ergodic by the previous argument, which shows that any ergodic measure would be conservative, which would then imply $\max (\la_\psi (\La_c), \la_{\psi\circ \i}(\La_c))>0$ and hence  $\la_\psi (\La_c)>0$ by the equivalence of (1) and (3). 
This completes the proof of the equivalence $(4)\Leftrightarrow (5)$.

These establish the equivalence of all (1)-(6).
To see that these are all equivalent to $ (7)$, we recall that
for any $\G$-critical $\psi$, the ergodicity of the $A$-action on $(\cal E_0,  m_{\la_\psi, \la_{\psi \circ \i}}|_{\cal E_0})$ is proved in \cite[Thm. 1.1]{LO2}.
The conservativity (resp. dissipativity) in (6)  and the conservativity in (7) (resp. dissipativity) are equivalent as the projection  $\cal E_0\to \Ga\ba G/M$ has the compact fiber, which is a closed subgroup between $M^\circ$ and $M$. Hence $(\cal E_0, A, m_{\la_\psi, \la_{\psi \circ \i}}|_{\cal E_0})$ is conservative only when $\psi$ is $\G$-critical by $(2)\Leftrightarrow (6)$.
This completes the proof of Theorem \ref{mp}.

We also show the following:
\begin{prop}
If $\psi$ is $\Ga$-critical,
then for any $(\la_\psi, \la_{\psi\circ \i})\in \cal M_\psi\times  \cal M_{\psi\circ\i}$,  the diagonal $\Ga$-action on $(\F \times \F ,  \lambda_\psi\otimes \lambda_{\psi\circ \i} )$ is ergodic and completely conservative.
\end{prop}
\begin{proof} By Theorem \ref{mp}, it suffices to show that
 $$(\la_\psi\times\la_{\psi\circ \op i})((\La \times\La ) - \La^{(2)})=0.$$
Set
$Q:=\La \times\La  - \La^{(2)}$
and $Q(x):=\{y\in\La : (x,y) \in Q\}$ for each $x\in \La$.
By Lemma \ref{lem.RA}(2),
we have
$$Q(x)=\{x\}\quad \text{ for all $x\in\La$}.$$
On the other hand, the conical property of an Anosov subgroup (Lemma \ref{lem.RA}(3)) implies that $\la_\psi$ is not atomic (Prop. 7.4 and Lem. 7.8 of \cite{LO}), and hence $\la_\psi(Q(x))=0$ for all $x\in\La$.
Therefore
\begin{equation}\label{eq.B}
(\la_\psi\times\la_{\psi\circ \op i})(Q)=\int_{x\in \La}\la_\psi(Q(x))\,d\la_{\psi\circ \op i}(x)=0,
\end{equation}
proving the proposition.
\end{proof}
\section{Growth indicator function and Lebesgue measure of $\La$}
We denote by $\rho$ the half sum of all positive roots of $(\fg, \fa)$.
A subset $\mathcal S$ of positive roots is called strongly orthogonal if
any any two distinct roots $\alpha, \beta$
in $\mathcal S$ are strongly orthogonal to each other, i.e., neither of
$\alpha\pm \beta$ is a root.
Let $\Theta$ denote the half sum of all roots in a maximal strongly orthogonal system of $(\fg, \fa)$; this does not depend on the choice of a maximal strongly orthogonal system (see \cite{Oh} where $\Theta$ is explicitly given for each simple algebraic group).

\begin{thm} \label{inf} Let $G$ be a connected simple real algebraic group of rank at least $2$. Let $\Ga<G$ be a discrete subgroup of infinite co-volume. 
Then
$$\psi_\Ga \le 2\rho -\Theta .$$
\end{thm}

\begin{proof}
This is proved by Quint \cite{Q3},
but the above explicit bound was not formulated, although his proof certainly gives that.  We give a slightly different and more direct proof for the sake of completeness.

Note that
 the right translation action of $G$
 on $\Ga\ba G$ gives a unitary representation
 $L^2(\Ga\ba G)$ with no non-zero fixed vector as $\Ga\ba G$ has infinite volume. 
We may then use \cite[Thm. 1.2]{Oh} which gives that
for any $K$-invariant functions
 $f\in L^2(\Ga\ba G)$, any $v\in \fa^+$, and any $\e>0$,
 \be\label{oh} \langle (\exp v) f, f\rangle \le d_\epsilon e^{-(1-\e) \Theta (v)} \|f\|_2^2\ee 
 where $d_\e>0$ depends only on $\e$. Therefore this theorem follows from Proposition \ref{theta0}.
\end{proof}

 \begin{prop}\label{theta0}
Suppose that there exists a function $\theta:\fa^+\to \br$ such that
for any $K$-invariant functions
 $f\in L^2(\Ga\ba G)$, any $v\in \fa^+$, and any $\e>0$,
 \be\label{oh} \langle (\exp v) f, f\rangle \le d_\epsilon e^{-(1-\e) \theta (v)} \|f\|_2^2\ee 
 where $d_\e>0$ depends only on $\e$.
 Then $$\psi_\Ga \le 2\rho -\theta .$$
\end{prop}
\begin{proof} 
 Fix $u\in \fa^+$ be a unit vector such that $\psi_\Ga(u)>0$.
 Fix an open cone $\cal C\subset \fa^+$
 containing $u$, and set $\cal C_T=\{v\in \cal C:
 \|v\|\le T\}$ and $B_{T}=K\exp (\cal C_T) K$ for each $T>1$.
 
 Define $$F_{T}(g, h):=\sum_{\gamma\in \Gamma}\mathbbm 1_{B_T} (g^{-1}\gamma h)$$ which we regard as a function on $\Ga\ba G\times \Ga\ba G$.
 Let $\e>0$.
 Let $U_\e=K U_\e K$ be a symmetric
 open neighborhood of $e$ which injects to $\Ga\ba G$ such that $U_\e B_T U_\e \subset   B_{T+\e} $
 for all $T>1$.
 Let $\Phi_\e$ be a non-negative $K$-invariant continuous function supported in $\Ga\ba \Ga U$ with $\int_{\Ga\ba G} \Phi_\e dx=1$.

 Let $$\eta=\eta_{\cal C}:=\sup \{|2\rho (v) -2\rho (u)| : v\in \cal C, \|v\|=1\}.$$
Using that for $g=k_1 (\exp v ) k_2$,
$dg=\Xi(v) dk_1 dv dk_2$ with 
$\Xi(v)
\asymp
e^{2\rho(v)}$,
and \eqref{oh},
we compute
\begin{align*}
    & \#\Ga\cap B_T =F_T(e,e)\\
     & \le \int_{\Ga\ba G \times \Ga\ba G}
     F_{T+\e} ([g], [h]) \Phi_\e ([g])\Phi_\e ([h]) dg dh
     \\
     &=\int_{\Ga\ba G\times\Ga\ba G}\sum_{\ga\in\Ga}\mathbbm{1}_{B_{T+\e}}(g^{-1}\ga h)\Phi_\e([g])\Phi_\e([h])\,dg\,dh\\
     &=\int_{\Ga\ba G}\int_{G}\mathbbm{1}_{B_{T+\e}}(g^{-1} h)\Phi_\e([g])\Phi_\e([h])\,dg\,dh\\
     &=\int_{\Ga\ba G}\int_{G}\mathbbm{1}_{B_{T+\e}}(g^{-1} )\Phi_\e([h]g)\Phi_\e([h])\,dg\,dh \\
     &=\int_{K\exp (\cal C_{T+\e}) K}\left(\int_{\Ga\ba G}\Phi_\e([h]g)\Phi_\e([h])\,dh\right)\,dg\\
     &\asymp
\int_{v\in \cal C_{T+\e}} \langle \exp v. \Phi_\e, \Phi_\e \rangle 
     e^{2\rho (v) }dv \\\
     &
     \le d_\e \int_{v\in \cal C_{T+\e}} 
     e^{(2\rho-(1-\epsilon) \theta)  (v) }dv 
     \cdot\norm{\Phi_\e}_{2}^2
     \\
     &\le  d_\e \int_0^{T+\e} \int_{v\in \cal C, \|v\|=1} 
     e^{(2\rho-(1-\epsilon) \theta)  (tv) } dv dt 
      \cdot\norm{\Phi_\e}_{2}^2
     \\ 
     & \ll   e^{(2\rho-(1-\epsilon) \theta)  ((T+\e) u)+ 2 (T+\e) \eta }
 \end{align*}
 
 where the implied constants are independent of $T>1$.
 Therefore
 $$\limsup_{T\to \infty} \frac{\log \# (\Ga\cap B_T)}T \le  (2\rho - \theta)(u) +\epsilon \theta(u) +2\eta.$$ 
 On the other hand,
 when $\psi_\Ga(u)>0$,
 $$\psi_\Ga(u)= \inf_{u\in \cal C} \limsup_{T\to \infty} \frac{\log \# (\Ga\cap K\exp (\cal C_T) K)}T$$ 
 where the infimum is taken over all open cones $\cal C$ containing $u$.
 Since $\eta=\eta_{\cal C}\to 0$ as $\cal C$ shrinks to the ray $\br_+u$,
 we get $$\psi_\Ga (u)\le  (2\rho - \theta)(u) +\epsilon \theta(u) .$$
 Since $\e>0$ was arbitrary,
 this implies
 $$\psi_\Ga (u)\le (2\rho-\theta)(u)$$
 as desired.
\end{proof}

\begin{rmk} \begin{enumerate}
    \item Corlette's theorem \cite{Co} shows a uniform gap theorem as above for rank one groups with property (T).
\item We remark that in a recent work \cite{KMO}, a stronger bound $\psi_\Ga\le \rho$ was conjectured for $\Ga$ Anosov.
\end{enumerate}
\end{rmk}

A connected simple real algebraic group is isomorphic to one of the following groups:
$\SO(n,1), \op{SU}(n,1), \op{Sp}(n,1), F_4$, which are groups of isometries
of real, complex, quarternionic hyperbolic spaces and the Cayley plane respectively.
If $X$ denotes the corresponding Riemannian
symmetric space as listed above, the Hausdorff dimension of $\partial X$ with respect to the Riemannian metric is given by
$k(n+1)-2$ where $k=1,2,4$, and $22$ respectively(\cite{Co}, \cite{Mi})
; they are equal to the volume entropy $D_X$ of $X$ with respect to a properly normalized Riemannian metric on $X$.

The following theorem is well-known due to the works of Sullivan (\cite{Su}, \cite{Su2}), Corlette \cite{Co} and Corlette-Iozzi \cite{CI}.
\begin{thm} \label{one} Let $G$ be a connected simple algebraic group of rank one. Let $\Ga<G$ be a convex cocompact subgroup such that $\Ga\ba G$ is not compact.
Then $$\dim_H(\La) <\dim_H(\partial X).$$
where $\op{dim}_H$ denotes the Hausdorff dimension with respect to the Riemannian metric on $\partial X$.
\end{thm}
\begin{proof} Let $\delta$ denotes the critical exponent of $\Ga$.
By \cite[Thm. 6.1, Cor. 6.2]{CI}, 
$\delta$ is equal to $\dim_H(\La)$ and the bottom, say, $\lambda_0$ of the $L^2$-spectrum of the negative Laplacian is given by $\delta( D_X-\delta).$ Now suppose that $\delta=D_X$.
By (\cite[Thm. 5.5]{Co}, \cite{Su2}), there exists a unique harmonic function on $\Ga\ba X$, and it is square-integrable. Since the constant function is a harmonic function,
it follows that $\Ga\ba X$ has finite volume, and hence compact,
as $\Ga$ is assumed to be convex cocompact. This proves the claim.
\end{proof}

We now deduce Theorem \ref{null} from Theorems \ref{mp} and \ref{inf}.
\subsection*{Proof of Theorem \ref{null}} Let $\G<G$ be Zariski dense and Anosov.
If $\text{rank} G=1$ and $\G<G$ is cocompact, then it is immediate that $\La=\mathcal F$.
We now suppose that $\Ga$ is not a cocompact lattice in a rank one group $G$. We claim that the Lebesgue measure of
$\La$ is zero. We write $G=G_1 G_2$ where $G_1$ is a product of all simple factors of rank one, and $G_2$ is a product of all simple factors of rank at least $2$.
Consider first the case when $G_2$ is trivial. Then $\Ga$ is of the form:
$\Ga=\big(\prod_{i=1}^k \pi_i \big)(\Sigma) $ where
$\Sigma$ is a Gromov hyperbolic group and $\pi_i$ is a convex cocompact representation of $\Sigma$ into a rank one simple factor of $G$. If $k=1$, it follows from Theorem \ref{one}.
If $k\ge 2$, then the Hausdorff dimension of 
$\La$ is at most the maximum of the Hausdorff dimension of the boundary of a rank one factor of $G$ (cf. proof of \cite[Theorem 3.1]{KMO2}); therefore it is strictly smaller than the Hausdorff dimension of $G/P$.
Hence the Lebesgue measure of $\La$ is zero.
Now suppose that $G_2$ is not trivial. Let $p:G\to G_2$ denote the canonical projection.  By the Anosov property of $\Ga$,
the projection $p(\Ga)<G_2$ is again an Anosov subgroup.
It suffices to prove that the limit set of $p(\Ga)$ has Lebesgue measure zero. Therefore, we may assume without loss of generality that $G=G_2$ and $G_2$ is simple.
Since $\Ga$ has infinite co-volume in $G$, as $\pi(\Ga)$ is Gromov hyperbolic,  it follows from Theorem \ref{inf} that the growth indicator  function $\psi_\Ga$ of $\Ga$ satisfies $\psi_\Ga<2\rho$, i.e., $2\rho$ is not $\Ga$-critical.
 Since the Lebesgue measure on $\F$
 is the $(G, 2\rho)$-conformal measure,  Theorem \ref{mp} implies the claim.

\begin{rmk}
Note that it is the consequence of Theorem \ref{null} that $\psi_\Ga<2\rho$ for all Anosov subgroups of $G$ which is not cocompact in $G$.
\end{rmk}

For a general discrete subgroup $\G<G$, we record the following:
\begin{prop}\label{ds}  If $\G<G$ is a discrete subgroup
with $\psi_\Ga<2\rho$, then the Lebesgue measure of the conical limit set $\La_c$ is zero.
In particular, if  $\G$ and $G$ are as in Theorem \ref{inf}, $\op{Leb}(\La_c)=0$.
\end{prop}
\begin{proof} 
If $\psi_\Ga <2\rho$, then
 $\sum_{\ga\in \Ga}e^{-2\rho (\mu(\ga))} <\infty$ by \cite[Lem. III 1.3]{Q4}.
 By \cite[Lem. 7.11]{LO} (Lemma \ref{divv}),
 this implies that 
 $\op{Leb}(\La_c)=0$. 
\end{proof}

\end{document}